\documentclass[12pt, a4paper]{amsart}

\usepackage{graphicx}
\usepackage{wrapfig}
\usepackage{comment}
\usepackage{lipsum}

\usepackage{pifont}

\usepackage[
unicode=true,
plainpages = false,
pdfpagelabels,
bookmarks=true,
bookmarksnumbered=true,
bookmarksopen=true,
breaklinks=true,
backref=false,
colorlinks=true,
linkcolor = DarkBlue,
urlcolor  = DarkBlue,
citecolor = DarkRed,
anchorcolor = green,
hyperindex = true,
hyperfigures
]{hyperref}
\hypersetup{
pdftitle={An explicit embedding of the group Q into a finitely presented group},
pdfauthor={V.H. Mikaelian},
pdfsubject={An explicit embedding of the group Q into a finitely presented group}
}

\usepackage[dvipsnames]{xcolor}
\colorlet{LightRubineRed}{RubineRed!70!}
\colorlet{Mycolor1}{green!10!orange!90!}
\definecolor{DarkRed}{HTML}{cc0000}
\definecolor{ChapterHeadColor}{HTML}{cc0000}
\definecolor{PartHeadColor}{HTML}{cc0000}
\definecolor{DarkBlue}{HTML}{0000cc}
\definecolor{QuoteColor}{HTML}{665665}

\usepackage[top=2.5cm, bottom=2.5cm, outer=2.5cm, inner=2.5cm, heightrounded
]{geometry}

\newcommand{\E}{{\mathcal E}}

\newcommand{\Z}{{\mathbb Z}}
\newcommand{\Q}{{\mathbb Q}}

\newcommand{\bigast}{\textrm{\footnotesize\ding{91}}}

\usepackage[bitstream-charter]{mathdesign}

\DeclareSymbolFont{cmsymbols}{OMS}{cmsy}{m}{n}
\SetSymbolFont{cmsymbols}{bold}{OMS}{cmsy}{b}{n}
\DeclareSymbolFontAlphabet{\mathcal}{cmsymbols}

\theoremstyle{plain}
\newtheorem{Theorem}{Theorem}[section]
\newtheorem{Lemma}[Theorem]{Lemma}
  
\newtheorem{Corollary}[Theorem]{Corollary}

\theoremstyle{definition}
\newtheorem{Definition}[Theorem]{Definition}

\theoremstyle{remark}
\newtheorem{Remark}[Theorem]{Remark} 
\newtheorem{Example}[Theorem]{Example} 

\numberwithin{equation}{section}

\usepackage[bitstream-charter]{mathdesign}

\DeclareSymbolFont{cmsymbols}{OMS}{cmsy}{m}{n}
\SetSymbolFont{cmsymbols}{bold}{OMS}{cmsy}{b}{n}
\DeclareSymbolFontAlphabet{\mathcal}{cmsymbols}

\begin{document}



\title[Auxiliary free constructions]{Auxiliary free constructions for explicit embeddings\\ of recursive groups}

\author{V. H. Mikaelian
}

\begin{abstract}
An auxiliary free construction  $\textstyle
\bigast_{i=1}^{r}(K_i, L_i, t_i)_M$ based on HNN-exten\-sions and on generalized free product of groups with amalgamated subgroups is suggested, and some of its basic properties are displayed. The proposed construction is a generalization of a series of structures used by Higman for embeddings of recursive groups into finitely presented groups, and also for certain structures we recently applied in the research on embeddings of recursive groups. 
Usage of this technical tool substantially simplifies some embedding methods for recursive groups.
A few technical results on specific subgroups in the suggested $\bigast$-constructions, in the HNN-extensions of groups, and in free products of groups with amalgamated subgroup are proved. The obtained properties are applied to build infinitely generated benign subgroups inside free groups of small rank.
\end{abstract}

\date{\today}

\maketitle

\setcounter{tocdepth}{1}

\let\oldtocsection=\tocsection
\let\oldtocsubsection=\tocsubsection
\let\oldtocsubsubsection=\tocsubsubsection
\renewcommand{\tocsection}[2]{\hspace{-12pt}\oldtocsection{#1}{#2}}
\renewcommand{\tocsubsection}[2]{\footnotesize \hspace{6pt} \oldtocsubsection{#1}{#2}}
\renewcommand{\tocsubsubsection}[2]{ \hspace{42pt}\oldtocsubsubsection{#1}{#2}}

{\footnotesize \tableofcontents}

\section{Introduction}
\label{SE Introduction}

\noindent
The first objective of this article is to present an auxiliary technical $\bigast$-construction:
\begin{equation}
\label{EQ *-construction short form} 
\textstyle
\bigast_{i=1}^{r}(K_i, L_i, t_i)_M
\end{equation}
which is a ``nested'' combination of HNN-extensions of groups and of generalized free products with amalgamated subgroups. 
The definition of \eqref{EQ *-construction short form} 
will be given in Section~\ref{SU Building the *-construction}, and some of its main properties and initial applications will follow  in sections
\ref{SU Subgroups in 
bigast}, 
\ref{SU Intersections and joins of benign subgroups},
\ref{SU Free products inside}, 
\ref{SU Benign subgroups in b c and a b c}. 

Working on embeddings of recursive groups we noticed that, although certain constructions in Higman's famous work \cite{Higman Subgroups of fP groups}, or elsewhere in the related literature, often look very diverse, many of them are noting but the particular implementations of the \textit{same} general structure that can be interpreted as \eqref{EQ *-construction short form}.  
Thus, it seems to be reasonable to obtain a few general basic properties of this sturcture in order to present many of arguments used in 
\cite{Higman Subgroups of fP groups,
A modified proof for Higman, 
On explicit embeddings of Q,
An explicit algorithm} and in some other papers, as trivial consequences of those basic features.

\medskip
The second motivation for this article is that in our research of the recent years \cite{
A modified proof for Higman, 
On explicit embeddings of Q,
An explicit algorithm,
Embeddings using universal words} many of the arguments rely on certain properties of particular subgroups in HNN-extensions of groups and in free products with amalgamated subgroups. 
Thus, in order to avoid repeated inclusion of specific auxiliary results and proofs in multiple papers, we stockpile them here, in particular in Chapter~\ref{SE Subgroups in free products with amalgamation and in HNN-extensions}, to use them via references in other articles. 
This seems to be especially appropriate in the present article, since most of those proofs are also necessary to study the $\bigast$-construction, and hence their inclusion here serves two affiliated purposes at once.

\subsection*{Acknowledgements}
\label{SU Acknowledgements}

In another study related to this article, we had the opportunities to discuss these constructions in detail with Prof.~O.~Bogopolski who, among other helpful points, also suggested to illustrate our most complex constructions with figures.

I am very much thankful to the SCS MES Armenia for the grant 25RG-1A187 which partially supports the current research.

\section{Definitions and references}

\subsection{Benign subgroups}
\label{SU Benign subgroups}

A key notion used by Higman to connect group-theo\-reti\-cal concepts to recursion and computability is that of benign subgroups:

\begin{Definition}
\label{DE benign subgroup}
A subgroup $H$ of a finitely generated group $G$ is called a \textit{benign subgroup} in $G$, if $G$ can be embedded in a finitely presented group $K$ with a finitely generated subgroup $L\le K$ such that $G \cap L = H$.
\end{Definition} 

In fact, \cite{Higman Subgroups of fP groups} suggests three definitions for benign subgroup, and then proves in Lemma~3.5 that they are equivalent. That approach is very comfortable for the purposes of \cite{Higman Subgroups of fP groups}, as it allows to employ any of three definitions depending on the specific suitable technical context. 
However, we restrict ourselves to the above definition only.

\smallskip
Whenever we have to work with more than one benign subgroups $H$ in $G$, we may specify the respective groups mentioned in Definition~\ref{DE benign subgroup} via $K=K_H$ and $L=L_H$, and then write $G \cap L_H = H$, to stress the correlation of the context with $H$, see Corollary~\ref{CO intersection and join are benign multi-dimensional}, 
Example~\ref{EX benign sample ONE},
figures~\ref{FI Figure_08_KI=KJ}, \ref{FI Figure_07_KJ}, etc.

An evident example of a benign subgroup is a finitely generated subgroup $H$ in a finitely presented group $G$. Then we just have to pick $K_H=G$ and $L_H=H$.
For examples of \textit{infinitely} generated benign subgroups see Chapter~\ref{SE Infinitely generated benign subgroups}.
Other details on benign subgroups see in chapters 3, 4 in \cite{Higman Subgroups of fP groups}, check also Chapter 3 in \cite{A modified proof for Higman}.

\subsection{Free constructions, references}
\label{SU Free constructions references}

Below we use three types of free constructions: 
generalized free products of groups with amalgamated subgroups,
HNN-extensions of groups by one or more stable letters, see \cite{Lyndon Schupp, 
Bogopolski,
Rotman, 
De La Harpe 2000}, 
and also the technical $\bigast$-construction which is a nested combination of both, see Chapter~\ref{SU The star construction}, and \textit{\nameref{SE Introduction}} above.

Here is the actual notation we are going to use for free constructions.
If the groups $G$ and $H$ have subgroups $A$ and $B$ respectively, and $\varphi : A \to B$ is an isomorphism, then we denote by $G*_{\varphi} H$ the generalized free product $\langle G,\, H \mathrel{|} a=a^{\varphi} \text{ for all $a\in A$}\, \rangle$ of $G$ and $H$ with subgroups $A$ and $B$ amalgamated with respect to $\varphi$. In the simple case, when in $G$ and $H$ two groups are intersecting in a subgroup $A$, we may take $B=A$ and assume $\varphi$ is the identical isomorphism on $A$. In such a case we prefer to write $G*_{A} H$.

If the group $G$ has subgroups $A$ and $B$, and $\varphi : A \to B$ is an isomorphism, then we denote
 by $G*_{\varphi} t$ the HNN-extension $\langle G, t \mathrel{|} a^t=a^{\varphi} \text{ for all $a\in A$}\, \rangle$ of the base group $G$ 
by the stable letter $t$
with respect to the isomorphism 
$\varphi$.
In case when $A=B$, and $\varphi$ is just the identity function on $A$, we prefer to write $G*_{A} t$. We may also say that the stable letter $t$ \textit{fixes} $A$ in $G$.
The HNN-extensions with more than one stable letters also are used. If for some pairs of subgroups $A_1, B_1;\; A_2,B_2;\,\ldots$ in $G$ we have the isomorphisms $\varphi_1: A_1 \to B_1,\; \varphi_2: A_2 \to B_2,\ldots$, then  we denote the HNN-extension $\langle G, t_1, t_2,\ldots \mathrel{|} a_1^{t_1}=a_1^{\varphi_1}\!\!,\; a_2^{t_2}=a_2^{\varphi_2}\!\!,\,\ldots\; \text{ for all $a_1\in A_1$, $a_2\in A_2,\ldots$}\, \rangle$ by $G *_{\varphi_1, \varphi_2, \ldots} (t_1, t_2, \ldots)$. 
%
Our notation for the \textit{normal form} in
free constructions (see Chapter~\ref{SE Subgroups in free products with amalgamation and in HNN-extensions} and Chapter~\ref{SU The star construction}) is very close to~\cite{Bogopolski}.

The definition for $\bigast$-construction \eqref{EQ *-construction short form} will be given in Chapter~\ref{SU The star construction} below.

\medskip
For general information on group theory we refer to the textbooks \cite{Robinson, Kargapolov Merzljakov, Rotman}.
And for more specific information on combinatorial group theory see \cite{Lyndon Schupp, Bogopolski, De La Harpe 2000}.

\subsection{Integer-valued sequences}
\label{SU Integer-valued sequences}

In \cite{Higman Subgroups of fP groups} the recursive functions are mostly used to operate over some integer-valued sequences interpreted as integer-valued functions.  
Let $\mathcal E$ be the set of all functions $f : \Z \to \Z$ with finite supports $\sup (f)=\big\{i\in \Z \mathrel{|} f(i)\neq 0\big\}$.
If for an $f$ there is a positive $m$ such that 
$\sup (f) \subseteq \{0,1,\ldots, m-1\}$, then 
denoting  $j_i = f(i)$ we get a \textit{sequence} $(j_0,\ldots,j_{m-1})$ which holds all the non-zero values of $f$. Since having the coordinates $f(i)$ of this sequence we can restore the function, we write it down via 
$f=(j_0,\ldots,j_{m-1})$.
Say,
$f=(0,0,5,-1,0,0,2,2)$
means that we can take $m=8$, and set
$f(2)=5$, $f(3)=-1$, 
$f(6)=f(7)=2$, and 
$f(i)=0$ for any index $i \in \Z \,{\backslash}\, \big\{2,3,6,7\big\}$.
Of course, $m$ is \textit{not} determined uniquely, and if needed, we may interpret the function $f$ as a \textit{longer} sequence by adding some extra zeros to it. Say, the above function $f$ can well be interpreted as the sequence $f=(0,0,5,-1,0,0,2,2,0,0,0,0)$ for $m=8+4=12$.
And the constant zero function $f(i)=0$ may be interpreted either as $f=(0)$ or, say, as $f=(0,0,0)$.
In the current article an application of such sequences can be found in Chapter~\ref{SE Infinitely generated benign subgroups}.

A significant part of \cite{Higman Subgroups of fP groups} consists of discussion about \textit{sets} of sequences of the above type with specific restrictions. Also, the \textit{Higman operations} $\iota, \upsilon, 
\rho,  \sigma, \tau, \theta, \zeta, \pi, \omega_m$ are being used to build new sets of sequences from the existing ones. See Chapter 2 in \cite{Higman Subgroups of fP groups} or the recent article \cite{The Higman operations and  embeddings} where the Higman operations are considered in detail. 
In fact, the Higman operations  are not used in this article and, therefore, we limit ourselves to this reference only.

\section{Subgroups in free products with amalgamation and in HNN-extensions}
\label{SE Subgroups in free products with amalgamation and in HNN-extensions}

\noindent
The below Lemma~\ref{LE NEW subgroups in amalgamated product} 
is a slight variation of Lemma 3.1 given by Higman in \cite{Higman Subgroups of fP groups} without a proof as a fact \textit{``obvious from the normal form theorem''}, and the next 
Lemma~\ref{LE NEW subgroups in HNN-extension}
is its analog for HNN-extensions. 
For the sake of completeness of our reasoning, we prefer to prove both lemmas, and to accompany them with corollaries
\ref{CO NEW G*H lemma corollary}, 
\ref{CO NEW G*H lemma corollary case A'=1},
\ref{CO NEW G*t lemma corollary},
\ref{CO NEW G*t lemma corollary case A'=1} 
which are going to often be used in \cite{A modified proof for Higman, 
On explicit embeddings of Q,
An explicit algorithm} 
and elsewhere.

\subsection{Subgroups in free products with amalgamation}
\label{SU Subgroups in free products with amalgamation}

\begin{Lemma}
\label{LE NEW subgroups in amalgamated product}
Let $\Gamma = G *_\varphi \!H$ be the free product of the groups $G$ and $H$ with amalgamated subgroups
$A \le G$ and $B \le H$
with respect to the isomorphism 
$\varphi: A \to B$.
If $G', H'$ respectively are subgroups of $G, H$, such that
for $A'=G'\cap A$ and $B'=H'\cap B$ 
we have  
$\varphi (A') = B'$, then for the subgroup $\Gamma'=\langle G',H'\rangle$  of $\Gamma$ and for the restriction $\varphi'$ of $\varphi$ on $A'$ we have:
\begin{enumerate}
\item 
\label{PO 1 LE NEW subgroups in amalgamated product}
$\Gamma' = G'*_{\varphi'} H'$,

\item 
\label{PO 2 LE NEW subgroups in amalgamated product}
$\Gamma' \cap A = A'$ and $\Gamma' \cap B=B'$,

\item 
\label{PO 3 LE NEW subgroups in amalgamated product}
$\Gamma' \cap G = G'$ and $\Gamma' \cap H = H'$.
\end{enumerate}
\end{Lemma}

\vskip-3mm
\begin{figure}[h]
	\includegraphics[width=390px]{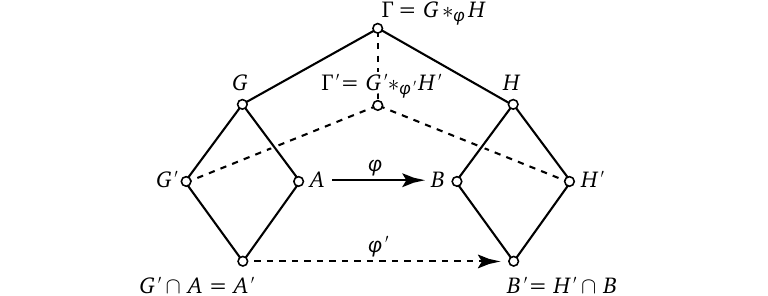}
	\caption{Construction of the group $\Gamma' = G'*_{\varphi'} H'$ in Lemma~\ref{LE NEW subgroups in amalgamated product}.}
	\label{FI Figure_02_Gamma_prime_1}
\end{figure}

\begin{proof}
By definition $\Gamma=(G * H)/N$ where 
$N$ is the normal closure of $\{ \varphi(a)\, a^{-1}\! \mathrel{|} a\in A\}$.
Any element in $\Gamma'$ can be presented as 
$c= c_0 \, c_1 \cdots c_{m-1}  c_m$
with each term 
$c_i$ picked from the factors $G'$ or $H'$ (the case $m=0$ is not ruled out). By where necessary merging the neighbour factors, we may suppose any two consecutive  terms $c_i, c_{i+1}$ are picked from different factors.
%
%
Fix a transversal $T_{A'}$ to $A'$ in $G'$, and a transversal $T_{B'}$ to $B'$ in $H'$, and apply to $c$ the ``analog'' of the procedure of bringing to a normal form. Namely,
for $c_m \!\!\in G'$, write $c_m = u \, l_n$ where $u \in A'$ and  $l_n  \in T_{A'}$ (ignore the value $n$ for now). 
Since $c_{m-1}\!\in H'$, then $c_{m-1}\,u=c_{m-1}  v \in H'$ for $v=\varphi(u)\in B'$, and so  $c_{m-1}  v = w \, l_{n-1}$  where $w \in B'$ and  $l_{n-1}  \in T_{B'}$. 
We already have the last two terms for $c= c_0  \,c_1  \cdots c_{m-2}\cdot w\,  l_{n-1} l_n$.
%
Continuing the process we 
get:
\begin{equation}
\label{EQ normal form of g}
c=l_0 \, l_1 \cdots l_{n-1} l_n
\end{equation}
where $n \le m$, \,
$l_0\in A'$ or $l_0\in B'$, and each of terms
$l_1,\ldots, l_n$ is a non-trivial element from $T_{A'}$ or  $T_{B'}$ such that no two consecutive  terms are from the same transversal.

To prove point \eqref{PO 1 LE NEW subgroups in amalgamated product} it is enough to show that \eqref{EQ normal form of g} is \textit{unique} for any $c\in \Gamma'$, because unique normal forms are one of the ways to define free products with amalgamation.
Notice that if $l_i, l_j$ from $T_{A'}$ are distinct modulo $A'$, they also are distinct modulo $A$ because $l_i,l_j$ are in $G'$, and so  $l_i^{\vphantom8}  l_j^{-1}\! \in A$ would imply 
$l_i^{\vphantom8}  l_j^{-1}\!\in (G' \cap A) = A'$. 
So we can choose a transversal $T_{A}$ to $A$ in $G$, containing $T_{A'}$ as its subset.
Similarly, if $l_i, l_j$ from $T_{B'}$ are distinct modulo $B'$, they also are distinct modulo $B$, and so we can choose a transversal $T_{B}$ to $B$ in $H$, containing $T_{B'}$. 
Finally, we may consider $l_0$ as an element from $A$ or $B$.
Thus, \eqref{EQ normal form of g} is nothing but the normal form of $c$ in $\Gamma = G *_\varphi H$ written using $T_A$ and $T_B$.
Since it is unique,  point \eqref{PO 1 LE NEW subgroups in amalgamated product} is proved.

Points \eqref{PO 2 LE NEW subgroups in amalgamated product}, \eqref{PO 3 LE NEW subgroups in amalgamated product} now follow from point \eqref{PO 1 LE NEW subgroups in amalgamated product}, and from uniqueness of the normal form.
\end{proof}

\begin{Corollary} 
\label{CO NEW G*H lemma corollary}
Let $\Gamma = G*_{A} H$,
and let $G'\le G$, $H'\le H$ be subgroups for which $G' \cap\, A = H' \cap\, A$. Then for $\Gamma'=\langle G', H'\rangle$ and $A' = G' \cap\, A $ we have:

\begin{enumerate}

\item 
\label{PO 1 CO NEW G*H lemma corollary}
$\Gamma'= G' \!*_{A'} H'$, in particular, if  $A\le G'\!,\,H'$, then $\Gamma'= G' *_{A} H'$;

\item 
\label{PO 2 CO NEW G*H lemma corollary}
$\Gamma' \cap\, A = A'$, in particular, if  $A\le G'\!,\,H'$, then $\Gamma' \cap\, A = A$\,;

\item 
\label{PO 3 CO NEW G*H lemma corollary}
$\Gamma' \cap\, G = G'$ and 
$\Gamma' \cap\, H = H'$.

\end{enumerate}
\end{Corollary}

\begin{figure}[h]
	\label{FI Figure_03_Gamma_prime_2}
	\includegraphics[width=390px]{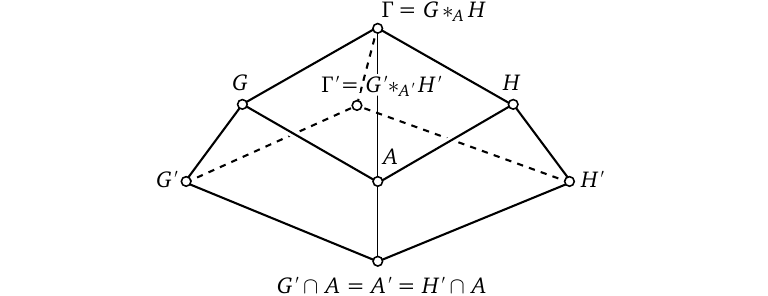}
	\caption{Construction of the group $\Gamma' = G'*_{A'} H'$ in Corollary~\ref{CO NEW G*H lemma corollary}.}
\end{figure}

In case the subgroup $A'$ above is \textit{trivial}, then we have much simpler situation:

\begin{Corollary} 
\label{CO NEW G*H lemma corollary case A'=1}
If in the notation of Corollary~\ref{CO NEW G*H lemma corollary} the subgroup $A'$ is trivial, then:

\begin{enumerate}

\item 
\label{PO 1 CO NEW G*H lemma corollary case A'=1}
$\Gamma'= G' * H'$ is the ordinary free product of $G'$ and $H'$;

\item 
\label{PO 2 CO NEW G*H lemma corollary case A'=1}
If also 
$G'\cong F_m$ and $H'\cong F_n$ are isomorphic to free groups of rank $m$ and $n$ respectively, then $\Gamma' \cong F_{m+n}$.

\end{enumerate}
\end{Corollary}

\subsection{Subgroups in HNN-extensions}
\label{SU Subgroups in HNN-extensions}

The facts of previous subsection have direct  analogs for the HNN-extensions:

\begin{Lemma}
\label{LE NEW subgroups in HNN-extension}
Let $\Gamma=G *_{\varphi} t$ be the HNN-extension of the base group $G$ by the stable letter $t$
with respect to the isomorphism 
$\varphi: A \to B$
of the subgroups 
$A, B \le G$.
If $G'$ is a subgroup of $G$  such that
for $A'=G'\cap\, A$ and $B'=G'\cap\, B$ 
we have  
$\varphi (A') = B'$, then for the subgroup $\Gamma'=\langle G',t\rangle$  of $\Gamma$ and for the restriction $\varphi'$ of $\varphi$ on $A'$ we have:
\begin{enumerate}
\item 
\label{PO1 LE NEW subgroups in HNN-extension}
$\Gamma' =G' *_{\varphi'} t$,

\item 
\label{PO2 LE NEW subgroups in HNN-extension}
$\Gamma' \cap G = G'$,

\item 
\label{PO3 LE NEW subgroups in HNN-extension}
$\Gamma' \cap \,A = A'$ \;and\; $\Gamma' \cap B = B'$.
\end{enumerate}
\end{Lemma}

\begin{figure}[h]
	\label{FI Figure_04_Gamma_HNN_prime_1}
	\includegraphics[width=390px]{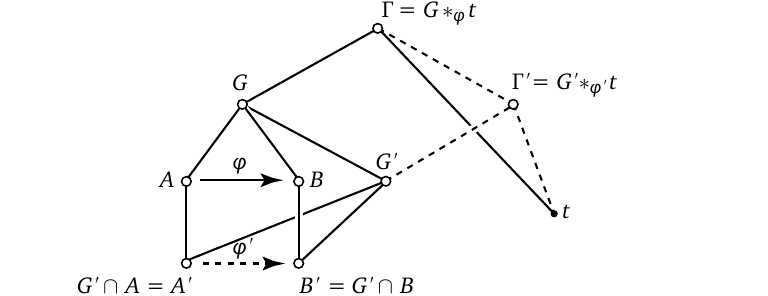}
	\caption{Construction of the group $\Gamma' = G'*_{\varphi'} t$ in Lemma~\ref{LE NEW subgroups in HNN-extension}.}
\end{figure}

\begin{proof}
By definition $\Gamma=(G * \langle t \rangle)/N$ where 
$N$ is the normal closure of $\{ \varphi(a)\, a^{-t}\! \mathrel{|} a\in A\}$.
Any element in $\Gamma'$ can be presented as 
$c= c_0 \, t^{s_{1}}  c_1 \, t^{s_{2}} \cdots \, t^{s_{m-1}} c_{m-1} \, t^{s_{m}} c_m \, t^{s_{m+1}}$
with each  
$c_i$ picked from $G'$
and $t^{s_{i}}$ picked from $\langle t \rangle$ 
(the case $m=0$ is not ruled out). By where necessary merging the neighbour factors, we may suppose 
no $c_i$ or $t^{s_{i}}$ are trivial except the last one, possibly.
%
%
Fix a transversal $T_{A'}$ to $A'$ in $G'$, and a transversal $T_{B'}$ to $B'$ in $G'$, and apply to $c$ the ``analog'' of the procedure of bringing to a normal form. Namely,
if, say, $s_{m}<0$, then write $t^{s_{m}} c_m 
= t^{s_{m}+1} t^{-1} u \, l_n
= t^{s_{m}+1}  u^t t^{-1} l_n$ where $u \in A'$ and  $l_n  \in T_{A'}$. 
Since $u\in A'$, then $u^t = \varphi(u)$ is some element $v \in B'$, and so 
$t^{s_{m}+1}  u^t t^{-1} l_n
=t^{s_{m}+1}  v \, t^{-1} l_n$.

If still $s_{m}\!+1<0$, we can repeat the step to continue to: 
$$t^{s_{m}+1}  v \, t^{-1} l_n
= t^{s_{m}+2} t^{-1} w\, l_{n-1} \, t^{-1} l_n
= t^{s_{m}+2} \, z t^{-1} l_{n-1} \, t^{-1} l_n
$$
with 
$w\in A'$ and $z=w^t \in B'$ (the cases $v=1$, $w=1$ or $z=1$ are not ruled out). After a few such steps our work on $t^{s_{m}} c_m$ will finished. In case $s_{m}>0$, we would use $B'$ and $T_{B'}$, and write $t^{s_{m}} c_m 
= t^{s_{m}-1} t\, u \, l_n
= t^{s_{m}-1}  u^{t^{-1}} t\, l_n$, etc., instead. 

Assume a leftover $h$ remains after we finish the job with $t^{s_{m}} c_{m}$. 
Concatenate $h$ to the the next syllable $t^{s_{m-1}} c_{m-1}$, and repeat all the above steps for 
$t^{s_{m-1}} (c_{m-1} h)$  (taking into account if $s_{m-1}$ is negative or positive). 

Finally, if during our process subwords of type $t^{-1} 1 \, t$ or $t\, 1 \, t^{-1}$ occur,  just cancel them out.
At the end of this process we 
get:
\begin{equation}
\label{EQ normal form in G*t}
c=l_0 \,t^{\varepsilon_1} l_1 \,t^{\varepsilon_2}
\cdots 
\,t^{\varepsilon_{n-1}}
l_{n-1} 
\,t^{\varepsilon_{n}}
l_n
\end{equation}
where $n \le m$, 
$\varepsilon_i=\pm 1$, 
$l_0\in G'$, and for $i=1,\ldots,n$
if $\varepsilon_i=- 1$, then $l_i\in T_{A'}$; while if $\varepsilon_i=1$, then $l_i\in T_{B'}$. The case $l_i=1$ is \textit{not} ruled out, so subsequences of type $t^{-1} 1\,  t^{-1} 1 = t^{-2}$ or $t 1\,  t 1 = t^2$ are possible, but subsequences $t^{-1} 1 \, t$ or $t\, 1 \, t^{-1}$ still are impossible in \eqref{EQ normal form in G*t}.

The product 
\eqref{EQ normal form in G*t} meets all the formal requirements on the normal form in HNN-exten\-sions, so to prove point \eqref{PO1 LE NEW subgroups in HNN-extension} it is enough to show that \eqref{EQ normal form in G*t} is \textit{unique} for any $c\in \Gamma'$ (unique normal forms are one of the ways to define HNN-extensions).
If $l_i, l_j$ from $T_{A'}$ are distinct modulo $A'$, they also are distinct modulo $A$ (see the proof of Lemma~\ref{LE NEW subgroups in amalgamated product}).
We can choose a transversal $T_{A}$ to $A$ in $G$, containing $T_{A'}$.
Similarly, we can choose a transversal $T_{B}$ to $B$ in $H$, containing $T_{B'}$. 

Thus, \eqref{EQ normal form in G*t} is unique as it is the normal form of $c$ in $\Gamma = G *_\varphi t$ written inside $T_A$ and $T_B$.
Points \eqref{PO2 LE NEW subgroups in HNN-extension}, \eqref{PO3 LE NEW subgroups in HNN-extension} now follow from point \eqref{PO1 LE NEW subgroups in HNN-extension}.
\end{proof}

\begin{Corollary} 
\label{CO NEW G*t lemma corollary}
Let $\Gamma = G*_{A} t$,
and let $G'\le G$ be a subgroup. Then for $\Gamma'=\langle G', t\rangle$ and $A' = G' \cap\, A $ we have:

\begin{enumerate}

\item 
\label{PO 1 CO NEW G*t lemma corollary}
$\Gamma'= G' \!*_{A'} t$, in particular, if  $A \le G'$, then $\Gamma'= G' *_{A} t$;

\item 
\label{PO 2 CO NEW G*t lemma corollary}
$\Gamma' \cap\, A = A'$, in particular, if  $A\le G'$, then $\Gamma' \cap\, A = A$\,;

\item 
\label{PO 3 CO NEW G*t lemma corollary}
$\Gamma' \cap\, G = G'$.

\end{enumerate}
\end{Corollary}

\begin{figure}[h]
	\label{FI Figure_05_Gamma_HNN_prime_2}
	\includegraphics[width=390px]{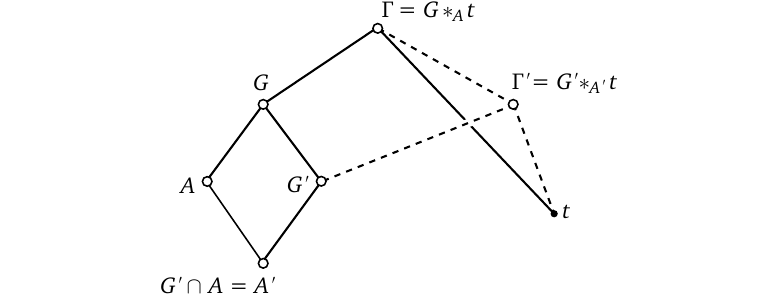}
	\caption{Construction of the group $\Gamma' = G'*_{A'} t$ in Corollary~\ref{CO NEW G*t lemma corollary}.}
\end{figure}

\begin{Remark} 
\label{RE about multiple stable letters}
It is easy to adapt the proof of Lemma~\ref{LE NEW subgroups in HNN-extension} for the case of multiple stable letters $t_1, \ldots, t_k$. And that adaptation will be especially simple, if all the stable letters $t_1, \ldots, t_k$ just fix the \textit{same} subgroup $A$ in $G$. In such a case, say, point~\eqref{PO 1 CO NEW G*t lemma corollary} in Corollary~\ref{CO NEW G*t lemma corollary} will read: 
$\Gamma'= G' \!*_{A'} (t_1, \ldots, t_k)$ for $\Gamma'=\langle G', t_1, \ldots, t_k\rangle$ and for $A' = G' \cap\, A$. 
\end{Remark}

In case the above $A'$  is \textit{trivial}, we have much simpler situation:

\begin{Corollary} 
\label{CO NEW G*t lemma corollary case A'=1}
If in the notation of Corollary~\ref{CO NEW G*t lemma corollary} the subgroup $A'$ is trivial, then:

\begin{enumerate}

\item 
\label{PO 1 CO NEW G*t lemma corollary case A'=1}
$\Gamma'= G' *\langle t \rangle$ is the ordinary free product of $G'$ and of an infinite cycle $\langle t \rangle$;

\item 
\label{PO 2 CO NEW G*t lemma corollary case A'=1}
If also 
$G'\cong F_m$ is isomorphic to free groups rank $m$, then $\Gamma' \cong F_{m+1}$.

\end{enumerate}
\end{Corollary}

\section{The $\bigast$-construction and its subgroups}
\label{SU The star construction}

\subsection{Building the $\bigast$-construction}
\label{SU Building the *-construction}

Let $G\le M  \le K_1,\ldots,K_r$ be an arbitrary system of groups  such that $K_i \cap  K_j=M$ for any distinct indices $i,j=1,\ldots,r$.
Picking in each $K_i$ a subgroup $L_i$ 
we can  build the ``nested'' free construction:
\begin{equation}
\label{EQ initial form of star construction}
\Big(\cdots
\Big( \big( (K_1 *_{L_1} t_1) *_M (K_2 *_{L_2} t_2) \big) *_M  (K_3 *_{L_3} t_3)\Big)\cdots 
\Big) *_M  (K_r *_{L_r} t_r)
\end{equation}
in which the HNN-extensions $K_i *_{L_i} t_i$ all are amalgamated in their common subgroup $M$.
To avoid the very bulky notation of \eqref{EQ initial form of star construction}, let us for the sake of briefness denote that construction via \eqref{EQ *-construction short form}, and  call it a\, $\bigast$-\textit{construction}.

\begin{figure}[h]
	\includegraphics[width=390px]{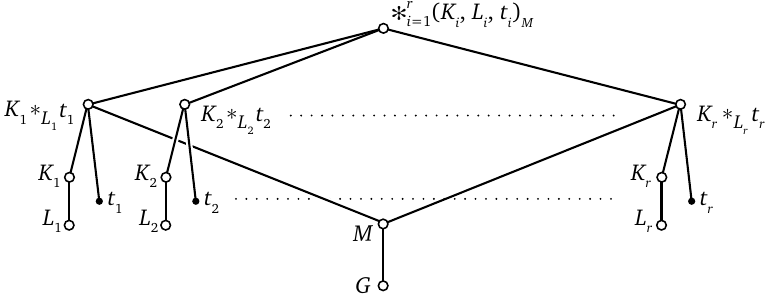}
	\caption{Construction of the group\, $\bigast_{i=1}^{r}(K_i, L_i, t_i)_M$ in \eqref{EQ initial form of star construction} and in \eqref{EQ *-construction short form}.} 
	\label{FI Figure_01_Star_construction}
\end{figure}

We are going to also use this group in some extreme cases, such as, $G=M$ or $K_i=M$ for all $i=1,\ldots,r$. 
Denote $G \cap \, L_i = A_i$, $i=1,\ldots,r$.
When for each $i$ we are limited to $K_i=G$, $L_i=A_i$, $M=G$, then 
$\bigast_{i=1}^{r}(G, A_i, t_i)_G$ is noting but the usual HNN-extension $G *_{A_1,\ldots,\,A_r } \!(t_1,\ldots,t_r)$.
And when, in addition to that, all the subgroups $A_i$ are trivial, then this $\bigast$-construction simply is the free product $G * \langle t_1,\ldots,t_r \rangle$ where $\langle t_1,\ldots,t_r \rangle\cong F_r$ is a free group of rank $r$.

\medskip
An evident feature of this construction is:

\begin{Lemma}
\label{LE finitey presented star construction}
In the above notation, if each of $K_i$ is finitely presented, while $M$ and each of $L_i$ are finitely generated, $i=1,\ldots,r$, then $\bigast_{i=1}^{r}(K_i, L_i, t_i)_M$ also is finitely presented. 
\end{Lemma}
 
We plan to employ this feature of $\bigast$-construction to build finitely presented overgroups of $G$, and Corollary~\ref{CO intersection and join are benign multi-dimensional} below is going to be the first usage of it, see also applications of that corollary in \cite{A modified proof for Higman, On explicit embeddings of Q,
An explicit algorithm}.

\subsection{Subgroups in $\bigast_{i=1}^{r}(K_i, L_i, t_i)_M$}
\label{SU Subgroups in 
bigast} 
It turns out that some HNN-extensions may be discovered inside $\bigast$-constructions.

\begin{Lemma}
\label{LE intersection in bigger group multi-dimensional}
If $G\le M  \le K_1,\ldots,K_r$ and $A_i  = G \cap \, L_i$ are the groups mentioned above, then in ${\bigast}_{i=1}^{r}(K_i, L_i, t_i)_M$
we have:
$$
\langle G, t_1,\ldots,t_r \rangle= 
G *_{A_1,\ldots,\,A_r } \!(t_1,\ldots,t_r).
$$
\end{Lemma} 

\begin{proof} 
Applying induction over $r$ we for $r=2$ have to display $\langle G, t_1,t_2 \rangle= 
G *_{A_1,\,A_2 } (t_1,t_2)$
in the $\bigast$-construction:
\begin{equation}
\label{EQ star for two groups}
(K_1 *_{L_1} t_1) *_M (K_2 *_{L_2} t_2).
\end{equation}
In $K_1 *_{L_1} t_1$ we by
Corollary~\ref{CO NEW G*t lemma corollary}\;\eqref{PO 1 CO NEW G*t lemma corollary} have
$\langle G, t_1 \rangle
= G *_{G \,\cap\, L_1} t_1
= G *_{A_1} t_1$. 
Similarly, $\langle G, t_2 \rangle= G *_{A_2} t_2$ in $K_2 *_{L_2} t_2$.
And  since in \eqref{EQ star for two groups} the intersection of both $\langle G, t_1 \rangle$ and $\langle G, t_2  \rangle$ with $G$ clearly is $G$, we apply  Corollary~\ref{CO NEW G*H lemma corollary}\;\eqref{PO 1 CO NEW G*H lemma corollary}
to get: 
$$
\langle G, t_1 , t_2 \rangle= 
\big\langle\langle G, t_1\rangle, \langle G, t_2\rangle\big\rangle= 
(G *_{A_1} t_1) *_G (G *_{A_2} t_2).
$$
But the above amalgamated free prodcut is noting but $G *_{A_1,\,A_2 } (t_1,t_2)$, which is trivial to see by listing all the defining relations of both constructions: relations of $G$ followed by relations stating that $t_1$ fixes the  $A_1$ and $t_2$ fixes  $A_2$ (plus the relations identifying both copies of $G$, if we initially assume them to be disjoint).

Next assume the statement is true for $r\!-\!1$, i.e.,
\begin{equation}
\label{EQ for inductive step in lemma}
\langle G, t_1,\ldots,t_{r-1} \rangle= 
G *_{A_1,\ldots,\,A_{r-1} } (t_1,\ldots,t_{r-1} ).
\end{equation}
Again by Corollary~\ref{CO NEW G*t lemma corollary}\;\eqref{PO 1 CO NEW G*t lemma corollary} write
$\langle G, t_r \rangle
= G *_{G \,\cap\, L_r}\!  t_r
= G *_{A_r}\! t_r$.
We have that \eqref{EQ for inductive step in lemma} and $\langle G, t_r \rangle$ both intersect with $G$ in $G$, and by Corollary~\ref{CO NEW G*H lemma corollary}\;\eqref{PO 1 CO NEW G*H lemma corollary}
we get:
\begin{equation*}
\label{EQ rewritten long form  multi-dimensional}
\langle G, t_1,\ldots,t_r \rangle=
\big(
G *_{A_1,\ldots,\,A_{r-1}}  (t_1,\ldots,t_{r-1} )\big) *_G  (G *_{A_r} t_r)
=
G *_{A_1,\ldots,\,A_{r} } (t_1,\ldots,t_{r} ).
\end{equation*}
\end{proof}

For any group $G$ and its subgroup $A$ 
the well known equality $G \cap G^t \!= A$ for the HNN-extension $G*_A t$  can be generalized to the following:

\begin{Lemma}
\label{LE intersection in HNN extension multi-dimensional}
Let $A_1,\ldots,\,A_r$ be any subgroups in a group $G$ with the intersection 
$I=\bigcap_{\,i=1}^{\,r} \,A_i$. 
Then in $G *_{A_1,\ldots,\,A_r} (t_1,\ldots,t_r)$ we have:
\begin{equation}
\label{EQ gemeral intersection in HNN}
\textstyle
G \cap G^{t_1 \cdots\, t_r}
= I,
\end{equation}
\end{Lemma}

\begin{proof}
Choose a transversal $T_{A_i}$ to $A_i$ in $G$, $i=1,\ldots,r$.
Take any $g\in G$, and show that if $g^{t_1 \cdots\, t_r}\in G$, then $g$ is inside each of $A_i$.
Write $g=a_1 l_1$ where $a_1\in A_1$ and $l_1 \in T_{A_1}$. In turn, $a_1$ can be written as $a_1=a_2 l_2$ where $a_2\in A_2$ and $l_2 \in T_{A_1}$. This process can be continued for $A_3,\ldots, A_r$ (the case when some of $a_i$ or $l_i$, $i=1,\ldots,r$, are trivial is \textit{not} ruled out).
Since the inverse $t_i^{-1}$ of the stable letter $t_i$ also fixes $A_i$, calculation of the normal form for $g^{t_1 \cdots\, t_r}$ can be started via the following steps:
\begin{equation}
\label{EQ calculations in HNN}
\begin{aligned}
g^{t_1 \cdots\, t_r}
&=t_r^{-1}\cdots t_1^{-1} a_1 l_1 \, t_1\cdots t_r \\
& =t_r^{-1}\cdots t_2^{-1} a_1 t_1^{-1} l_1 \, t_1\cdots t_r 
\\
& =  \cdots \cdots \cdots \cdots \cdots \cdots \cdots \cdots \cdots \cdots \cdots   \\
& =  a_{r}\, t_{r}^{-1}  l_{r}\, t_{r-1}^{-1} l_{r-1} \cdots \,l_3\,t_2^{-1} l_2\, t_1^{-1} l_1 \, t_1\cdots t_r 
\end{aligned}
\end{equation}
The above belongs to $G$ only if it contains no stable letters $t_i$.
But the last line of \eqref{EQ calculations in HNN} has a $t_1$  in the syllabus $ t_1^{-1} l_1 \, t_1$ \textit{only}. Hence, that line does not contain $t_1$ only when $l_1=1$, i.e., $t_1^{-1} l_1 \, t_1=1$, and
$t_2^{-1} l_2 t_1^{-1} l_1 \, t_1\, t_2=t_2^{-1} l_2 \, t_2$.
Then to exclude $t_2$ we must have $l_2=1$ and $t_2^{-1} l_2 \, t_2=1$. 
At the end we get the last line of \eqref{EQ calculations in HNN} reduced to 
$a_{r} t_{r}^{-1} l_{r} t_{r}=a_{r}$ where $l_{r}=1$, and therefore $a_{r} \in I$.

On the other hand, any  $g \in I$ is fixed by each of $t_i$, and so $g^{t_1 \cdots t_r}=g\in G$, and thus, 
 $I \subseteq G \cap G^{t_1 \cdots t_r}$.
\end{proof}

Another proof of this lemma could be deduced from Corollary~\ref{CO NEW G*H lemma corollary} and 
Corollary~\ref{CO NEW G*t lemma corollary} (in a manner rather similar to the proof of Lemma~\ref{LE join in HNN extension multi-dimensional} below), but we prefer this version as it follows from some  trivial \textit{combinatorial} manipulations already.

\begin{Lemma}
\label{LE join in HNN extension multi-dimensional}
Let $A_1,\ldots,\,A_r$ be any subgroups in a group $G$ with the join
$J=\big\langle\bigcup_{\,i=1}^{\,r} \,A_i\big\rangle$. 
Then in $G *_{A_1,\ldots,\,A_r} (t_1,\ldots,t_r)$ we have:
\begin{equation}
\label{EQ gemeral intersection in HNN multi-dimensional}
\textstyle 
G \cap \big\langle 
  \bigcup_{\,i=1}^{\,r} \,G^{t_i} \big\rangle
=J.
\end{equation}
\end{Lemma}

\begin{proof}
For simplicity write the proof for the case $r\!=\!3$, i.e.,  $J=\langle \, A_1, A_2, A_3\rangle$.
By Lemma~\ref{LE intersection in bigger group multi-dimensional} we have a $\bigast$-construction:
$$
G *_{A_1,\,A_2,\, A_3} (t_1,t_2,t_3)
= \big( (G *_{A_1} t_1)
*_G
(G *_{A_2} t_2)\big) 
*_G
(G *_{A_3} t_3).
$$
$G *_{A_1} t_1$ contains 
$G *_{A_1} \! G^{t_1}$, and in this subgroup we  have $J\! \cap A_1 = A_1$ and $G^{t_1}\! \cap\, A_1 = A_1$, and so 
by \eqref{PO 3 CO NEW G*H lemma corollary} in Corollary~\ref{CO NEW G*H lemma corollary} have $\langle J, G^{t_1}\rangle \cap G =  J$.
For the same reason $\langle  J, G^{t_2}\rangle \cap G =  J$.

Noticing 
$\langle  J, G^{t_1}\!, G^{t_2} \rangle = \big\langle 
\langle  J, G^{t_1} \rangle,
\langle  J, G^{t_2} \rangle
\big\rangle$
and applying to it 
Corollary~\ref{CO NEW G*H lemma corollary}\;\eqref{PO 2 CO NEW G*H lemma corollary} inside the group $(G *_{A_1} t_1)
*_G
(G *_{A_2} t_2)$ 
we have $\langle J, G^{t_1}\!, G^{t_2} \rangle \cap G = J$.
Since also $\langle J, G^{t_3}\rangle \cap G = J$, we again by Corollary~\ref{CO NEW G*H lemma corollary}\;\eqref{PO 2 CO NEW G*H lemma corollary} have 
$
\big\langle
\langle J, G^{t_1}\!, G^{t_2} \rangle, \langle J, \,G^{t_3}\rangle 
\big\rangle \cap G = J.
$
But since $J \le \langle 
G^{t_1}\!, G^{t_2}\!, G^{t_3}\rangle $, it remains to notice that
$\big\langle
\langle J, G^{t_1}\!, G^{t_2} \rangle, \langle J, \,G^{t_3}\rangle 
\big\rangle= \langle 
G^{t_1}\!, G^{t_2}\!, G^{t_3}\rangle$.
\end{proof}

In view of 
Lemma~\ref{LE intersection in bigger group multi-dimensional}, the analogs of 
Lemma~\ref{LE intersection in HNN extension multi-dimensional} and 
Lemma~\ref{LE join in HNN extension multi-dimensional} also hold inside a suitable $\bigast$-construction $\bigast_{i=1}^{r}(K_i, L_i, t_i)_M$ with appropriate 
$M, K_i,L_i$, $i=1,\ldots,r$, as we will see shortly.

\subsection{Intersections and joins of benign subgroups}
\label{SU Intersections and joins of benign subgroups}

The $\bigast$-construction of \eqref{EQ *-construction short form} together with 
Lemma~\ref{LE intersection in bigger group multi-dimensional},
Lemma~\ref{LE intersection in HNN extension multi-dimensional} and 
Lemma~\ref{LE join in HNN extension multi-dimensional}
now provide a corollary which will be extensively used below, in \cite{A modified proof for Higman, 
On explicit embeddings of Q,
An explicit algorithm} 
and elsewhere.

\begin{Corollary}
\label{CO intersection and join are benign multi-dimensional}
If the subgroups $A_1,\ldots,\,A_r$ are benign in a finitely generated group $G$, then:
\begin{enumerate}
\item 
\label{PO 1 CO intersection and join are benign multi-dimensional}
their intersection $I=\bigcap_{\,i=1}^{\,r} \,A_i$ also is benign in $G$;
\item 
\label{PO 2 CO intersection and join are benign multi-dimensional}
their join $J=\big\langle\bigcup_{\,i=1}^{\,r} \,A_i\big\rangle$ also is benign in $G$.
\end{enumerate}
Moreover, if the finitely presented groups $K_i$ with their finitely generated subgroups $L_i$ can be 
given for each $A_i$ explicitly, then the respective finitely presented overgroups $K_I$ and $K_J$ with finitely generated  subgroups $L_I$ and $L_J$
can also be given for $I$ and for $J$ explicitly.
\end{Corollary}

\begin{figure}[h]
	\includegraphics[width=390px]{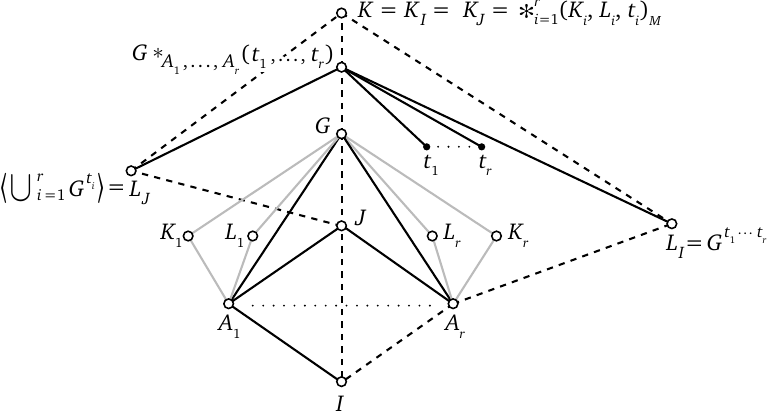}
	\caption{Construction of the group $K$ in Corollary~\ref{CO intersection and join are benign multi-dimensional}.} 
	\label{FI Figure_08_KI=KJ}
\end{figure}

\begin{proof}
As a finitely presented overgroup $K$ both for the above $I$ and for $J$ one may take the group $K=K_I=K_J=\textstyle{\bigast}_{i=1}^{r}(K_i, L_i, t_i)_M$ from \eqref{EQ *-construction short form}, see Lemma~\ref{LE finitey presented star construction}. 

For the intersection $I$ as a finitely generated subgroup we can take $L=L_I= G^{t_1 \cdots\, t_r}$. Then
$G$ and $L$ both are inside $\langle G, t_1,\ldots,t_r \rangle$ which is equal to $G *_{A_1,\ldots,\,A_r } \!(t_1,\ldots,t_r)$, i.e., to $\bigast_{i=1}^{r}(G, A_i, t_i)_G$
by Lemma~\ref{LE intersection in bigger group multi-dimensional}. And in the latter group 
$G \cap L = I$ holds by Lemma~\ref{LE intersection in HNN extension multi-dimensional}.

For the join $J$ as a finitely generated subgroup of $K$ take $L=L_J= \big\langle 
\bigcup_{\,i=1}^{\,r} \,G^{t_i} \big\rangle$. Then the groups
$G,L$ again are inside 
$$\textstyle 
\langle G, t_1,\ldots,t_r  \rangle=G *_{A_1,\ldots,\,A_r } \!(t_1,\ldots,t_r)={\bigast_{i=1}^{r}}(G, A_i, t_i)_G$$
by Lemma~\ref{LE intersection in bigger group multi-dimensional}. And in the latter  
$G \cap L = J$ holds by Lemma~\ref{LE join in HNN extension multi-dimensional}.
\end{proof}

\subsection{Free products inside $\bigast$-constructions}
\label{SU Free products inside} 
The following corollary allows to detect some free products inside $\bigast_{i=1}^{r}(K_i, L_i, t_i)_M$ and inside its subgroups, such as $G  *_{A_1,\ldots,\,A_r } \!(t_1,\ldots,t_r)$, whenever certain ``narrower'' free products are known inside $G$:

\begin{Corollary}
\label{CO smaller free product to the larger free product}
Let $A_1,\ldots,\,A_r$ be any subgroups in a group $G$ such that their join $J$ in $G$ is isomorphic to their free product
$\prod_{i=1}^{r} \,A_i$.  
Then the join
$\big\langle 
\bigcup_{\,i=1}^{\,r} G^{t_i} \big\rangle$
is isomorphic to the free product $\prod_{i=1}^{r} G^{t_i}$
in $G *_{A_1,\ldots,\,A_r } \!(t_1,\ldots,t_r)$, and hence in $\bigast_{i=1}^{r}(K_i, L_i, t_i)_M$.
\end{Corollary} 

\begin{proof}
Since, for each $i=1,\ldots,r$ we have $\langle G, G^{t_i}\rangle = G *_{A_i} G^{t_i}$\!, and since all such 
subgroups intersect in $G$, and they together generate 
$G^* = \big\langle\! G,\,  G^{t_1}\!, \dots , G^{t_r} \! \big\rangle$, we see that $G^*$ is the free product of the groups $G *_{A_1} G^{t_1} ,\ldots, G *_{A_r} G^{t_r}$ all amalgamated in $G$. 

Thus, as generators and defining relations of $G^*$ we by default can list the following:
any set of generators in $G$ 
plus their copies in each $G^{t_i}$;
the relations $R$ for $G$, plus their copies $R^{t_i}$ for each $G^{t_i}$\!, plus the relations stating that  the copies of $G$ in each $G *_{A_i} G^{t_i}$ coincide, plus the relations stating that $a_i^{t_i}=a_i$ for arbitrary $a_i$ in each of $A_i$.

For each $i$ the map $\varphi_i:A_i \to G^{t_i}$ given by the rule
$\varphi_i:a_i \to a_i^{t_i}\in G^{t_i}$ is an isomorphism from  $A_i$ onto $A_i^{t_i}$.
Since $J$ is the \textit{free} product
$\prod_{i=1}^{r} \,A_i$, there exists a common continuation $\varphi$ (for all $\varphi_i$) from $J$ to the free product $\prod_{i=1}^{r} G^{t_i}$ (\textit{onto} its subgroup $\prod_{i=1}^{r} A^{t_i}$). 
Using this $\varphi$ we can construct the free product  $G^{**} = G *_\varphi \prod_{i=1}^{r} G^{t_i}$ with subgroups $J$ and $\prod_{i=1}^{r} A^{t_i}$ amalgamated by $\varphi$. 

As generators and defining relations of $G^{**}$ we can list the following: the generators earlier chosen for $G$ along with the relations $R$;
the copies of those generators in each $G^{t_i}$ and the copies $R^{t_i}$ of relations $R$;
plus the relations stating that $\varphi:J \cong \prod_{i=1}^{r} A^{t_i}$.
The latters can well be replaced by the relations stating that each $A_i$ in $J$ coincides with $A^{t_i}$ by $\varphi_i$.
But this is noting but the list we provided for $G^*$ earlier, and hence $G^*= G^{**}$.  Clearly, $G^{**}$ contains the free product $\prod_{i=1}^{r} G^{t_i}$ together with the subgroup $\prod_{i=1}^{r} A^{t_i}$  of the latter. 
\end{proof}

\begin{Remark}
Notice that the condition of this corollary about \textit{free} product of $A_i$ is relevant, and it cannot be omitted. Say, for the group $G *_{G,G}(t_1, t_2)$ 
we have $A_1=A_2=G$, and $A_1$, $A_2$ certainly do not generate their free product (unless $G$ is trivial). 
Then $g^{t_1}=g=g^{t_2}$ for any $g\in G$, and hence the subgroup $\big\langle\!   G^{t_1}, G^{t_2} \! \big\rangle = G$ is \textit{not} isomorphic to the free product $G^{t_1}* G^{t_2}$.
\end{Remark}

\section{Infinitely generated benign subgroups in free groups of small rank}
\label{SE Infinitely generated benign subgroups}

\subsection{The elements $b_i, b_f, a_f$}
\label{SU the elements bi bf af} 

Fix a free group $\langle a,b,c \rangle$ of rank $3$, and for each integer $i\in \Z$ denote $b_i=b^{c^i}$\!.  Then for each sequence $f=(j_0,\ldots,j_{m-1})\in \E$, see Section~\ref{SU Integer-valued sequences}, define the following elements $b_f$ and $a_f$ in $\langle a,b,c \rangle$:
\begin{equation}
\label{EQ b_f a_f h_f g_f definition}
\begin{split}
\text{
$b_f=b_{0}^{j_0} \cdots b_{m-1}^{j_{m-1}}
\quad \text{ and } \quad 
a_f=a^{b_f} = b_f^{-1} \cdot a \cdot  b_f$.
}
\end{split}
\end{equation}
For example, if $f=(5,2,-7,1)$ then 
$b_f = b^5\, (b^2)^{c}\, (b^{-7})^{c^2} \,(b)^{c^3}$ and: 
$$
a_f = a^{b^5 (b^2)^{c} (b^{-7})^{c^2} (b)^{c^3}}\!= 
\Big(b^5\, (b^2)^{c}\, (b^{-7})^{c^2}\, (b)^{c^3}\Big)^{\!-1}\!\!\!\!
\cdot a \cdot b^5\, (b^2)^{c}\, (b^{-7})^{c^2}\, (b)^{c^3}\!\!.
$$
Further, for any subset $\mathcal X$ of $\E$ denote $A_{\mathcal X}=\langle a_f \mathrel{|} f\!\in \mathcal X \rangle$. The products of type \eqref{EQ b_f a_f h_f g_f definition} are correctly defined for each $f\in \mathcal X$, as the sequences of $\E$ all have \textit{finite} supports only.

\subsection{The isomorphisms $\xi_m$ and $\xi'_m$}
\label{SU The isomorphisms ksi} 
Using the notation of the previous point, for any integer $m$ a pair of isomorphisms $\xi_m$ and $\xi'_m$  can be defined on the free group $\langle b,c \rangle\cong F_2$:
\begin{equation}
\label{EQ definition of xi_m}
\text{
$\xi_m(b)=b_{-m+1}\quad
\xi'_m(b)=b_{-m}$
\quad and \quad
$\xi_m(c)=\xi'_m(c)=c^2$\!.
}
\end{equation}
It is easy to verify that for any $i$:
$$
\xi_m(b_i)=b_{2i-m+1}, 
\quad{\rm and}\quad\;
\xi'_m(b_i)=b_{2i-m}\,.
$$
For this value of $m$ fix a couple of stable letters $t_m, t'_m$ to define an HNN-extension: 
$$
\Xi_m = \langle b,c \rangle *_{\xi_m, \xi'_m} (t_m, t'_m).
$$

\begin{figure}[h]
	\includegraphics[width=390px]{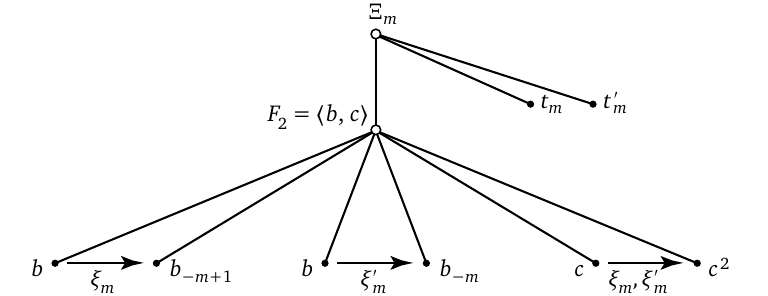}
	\caption{Construction of the group\, $\Xi_m$.} 
	\label{FI Figure_06_Xi_m}
\end{figure}

\begin{Lemma}
\label{LE Ksi}
For any $m$ in the above notation the following equalities hold in $\Xi_m$:
\begin{equation}
\label{EQ Ksi two equality}
\begin{split}
\langle b,c \rangle \cap \langle b_m, t_m, t'_m\rangle &= \langle b_m, b_{m+1},\ldots\rangle,
\\
\langle b,c \rangle \cap \langle b_{m-1}, t_m, t'_m\rangle &= \langle b_{m-1}, b_{m-2},\ldots\rangle.
\end{split}
\end{equation}
\end{Lemma}

\begin{proof}
We are going to prove only the first equality. 
For any $m$ and $i$ we have 
$b_{i}^{t_m} = \xi_m(b_i)=b_{2i-m+1}$ and 
$b_{i}^{t'_m} = \xi'_m(b_i)=b_{2i-m}$, and from here we get the following list for actions of $t_m$ and $t'_{m}$ on the elements $b_i$:
\begin{equation}
\label{EQ xi action}
\begin{split}
\ldots
b_{m-2}^{t_m} \!=\! b_{m-3},\;\;\;
b_{m-1}^{t_m} \!=\! b_{m-1},\;\;\;
b_{m}^{t_m}   \!=\! b_{m+1},\;\;\;
b_{m+1}^{t_m} \!=\! b_{m+3},\;\;\;
b_{m+2}^{t_m} \!=\! b_{m+5},
\ldots\\
\ldots
b_{m-2}^{t'_m} \!=\! b_{m-4},\;\;\;
b_{m-1}^{t'_m} \!=\! b_{m-2},\;\;\;
b_{m}^{t'_m}   \!=\! b_{m},\;\;\;
b_{m+1}^{t'_m} \!=\! b_{m+2},\;\;\;
b_{m+2}^{t'_m} \!=\! b_{m+4},\ldots
\;\;
\end{split}
\end{equation}
The \textit{inverse} actions of $t_m^{-1}$ and of ${t'}_{\!\!\!m}^{\,-1}$ can be understood from the list above by just ``swapping'' its two rows.

From \eqref{EQ xi action} it is straightforward that each of $b_m, b_{m+1},\ldots$ indeed is in $\langle b_m, t_m, t'_m\rangle$. 
For example, 
$b_{m+8}
=b_{m+4}^{t'_m}
=b_{m+2}^{{t'}_{\!\!\!m}^2}
=b_{m+1}^{{t'}_{\!\!\!m}^3}
=b_{m}^{t_m \cdot {t'}_{\!\!\!m}^3}\in \langle b_m, t_m, t'_m\rangle$. 

And on the other hand, bringing any word $w$ on letters $b_m, t_m, t'_m$ to the normal form in 
$\Xi_m$
we first have to  do cancellations like $t_m^{-1} b_m t_m=b_{m+1}$, and   ${t'}_{\!\!\!m}^{\,-1} b_{m} {t'}_{\!\!\!m}=b_{m}$. Repeated applications of such steps may create in $w$ some new letters 
$b_m, b_{m+1},\ldots$ so that we may also have to do ``reverse'' cancellations like 
$t_m b_{m+1} t_m^{-1}=b_{m}$, 
\; 
$t_m b_{m+3} t_m^{-1}=b_{m+1}$, etc...
or \,
$t'_m b_m {t'}_m^{-1}=b_{m}$,
\;
$t'_m b_{m+2} {t'}_m^{-1}=b_{m+1}$, etc...

As we see, bringing $w$ to normal form we \textit{never} get any $b_i$ outside $\langle b_m, b_{m+1},\ldots\rangle$. If, in addition, $w$ is in $\langle b,c \rangle$, then the normal form we obtained should contain no letters 
$t_m^{\pm 1}$ or ${t'}_{\!\!\!m}^{\pm 1}$\!.\,
That is, if $w$ is in $\langle b,c \rangle$, it in fact is in $\langle b_m, b_{m+1},\ldots\rangle$, and we have $\langle b,c \rangle \cap \langle b_m, t_m, t'_m\rangle = \langle b_m, b_{m+1},\ldots\rangle$.
\end{proof}

Since the free group $\langle a,b,c \rangle$ can be written as a free product $\langle a \rangle * \langle b,c \rangle$, and since $\Xi_m$ does not involve any $a$, then
using the first of the equalities \eqref{EQ Ksi two equality} in Lemma~\ref{LE Ksi} (or using uniqueness of representations of elements in free products)
we in the free product $\Theta=\langle a \rangle * \,\Xi_m$ have: 

\begin{Lemma}
\label{LE Ksi for G}
For any $m$ in the above notation the following equalities hold in $\Theta=\langle a \rangle * \,\Xi_m$:
\begin{equation}
\label{EQ a simple example of bening subgroup}
\begin{split}
G \cap \langle b_m, t_m, t'_m\rangle = \langle b_m, b_{m+1},\ldots\rangle
\;\;\; {\it and} \;\;\;  
G \cap \langle a, b_m, t_m, t'_m\rangle = \langle a, b_m, b_{m+1},\ldots\rangle,
\hskip3mm \\
G \! \cap \!\langle b_{m-1}, t_m, t'_m\rangle  \! =  \! \langle b_{m-1}, b_{m-2},\ldots\rangle
\;\; {\it and} \;\;\,  
G \! \cap \! \langle a, b_{m-1}, t_m, t'_m\rangle \!  =  \! \langle a, b_{m-1}, b_{m-2},\ldots\rangle.
\end{split}
\end{equation}
\end{Lemma}

\subsection{Benign subgroups in $\langle b,c \rangle$ and in $\langle a, b,c \rangle$}
\label{SU Benign subgroups in b c and a b c} 

Lemma~\ref{LE Ksi} and Lemma~\ref{LE Ksi for G} display some examples of infinitely generated benign subgroups in $\langle b,c \rangle$ and in $\langle a, b,c \rangle$:

\begin{Corollary}
\label{CO two types of benign subgruops in <b,c>}
Any of its subgroups of the following two types:
$$
\langle b_m, b_{m+1},\ldots\rangle,
\quad
\langle b_{m-1}, b_{m-2},\ldots\rangle.
$$
is benign in $\langle b,c \rangle$ for any $m$.
\end{Corollary}

\begin{Corollary}
\label{CO four types of benign subgruops in G}
Any of its subgroups of  the following  four types:
$$
\langle b_m, b_{m+1},\ldots\rangle,
\quad
\langle a, b_m, b_{m+1},\ldots\rangle,
\quad
\langle b_{m-1}, b_{m-2},\ldots\rangle,
\quad
\langle a, b_{m-1}, b_{m-2},\ldots\rangle
$$
is benign in $\langle a,b,c \rangle$ for any $m$.
\end{Corollary}

Moreover, applying the $\bigast$-construction given in \eqref{EQ *-construction short form},
and later used in the proof of
Corollary~\ref{CO intersection and join are benign multi-dimensional}, we can merge the subgroups of the above types to get further samples of benign subgroups in $\langle b,c \rangle$ or in $\langle a, b,c \rangle$. Here is an example of application of this idea:

\begin{figure}[h]
	\includegraphics[width=390px]{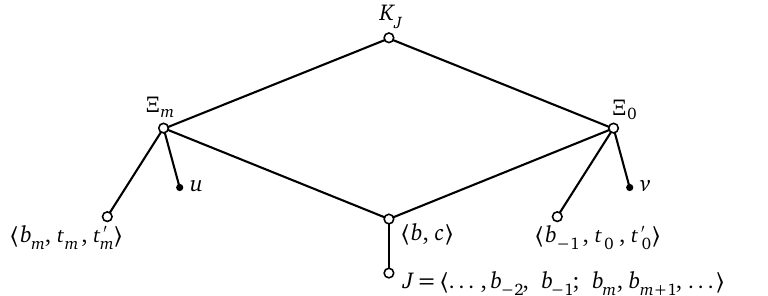}
	\caption{Construction of the group $K_J$ in Example~\ref{EX benign sample ONE}.} 
	\label{FI Figure_07_KJ}
\end{figure}

\begin{Example}
\label{EX benign sample ONE}
The subgroup\, 
$\langle \ldots b_{-2}, b_{-1};\;
b_m, b_{m+1},\ldots\rangle$\,
is benign in $\langle b,c \rangle$ and in $\langle a, b,c \rangle$ for arbitrary (non-negative) integer $m$.
Indeed, first build the $\bigast$-construction:
\begin{equation}
\label{EQ benign sample ONE}
K_J =
\left(\Xi_m *_{\langle b_m, t_m, t'_m\rangle} u \right)
\, *_{\langle b,c\rangle}
\left(\Xi_0 *_{\langle b_{-1}, t_0, t'_0\rangle} v \right)
\end{equation}
from Chapter~\ref{SU The star construction} for the groups
$K_1 = \Xi_m$,\;
$K_2 = \Xi_0$,\;
$M = \langle b,c\rangle$.
The group $K_J$ clearly is finitely presented, see Lemma~\ref{LE finitey presented star construction}. Then by Lemma~\ref{LE intersection in bigger group multi-dimensional} our $K_J$ contains the subgroup: 
$$
\big\langle \langle b,c\rangle,\, u, v \big\rangle
=
\langle b,c\rangle *_{\langle b_m, b_{m+1},\ldots  \rangle,\;\langle b_{-1}, b_{{-2}},\ldots  \rangle} (u,v).
$$
This subgroup may no longer be finitely presented, as $\langle b_m, b_{m+1},\ldots  \rangle$ and $\langle b_{-1}, b_{{-2}},\ldots  \rangle$ are not finitely generated. But by Lemma~\ref{LE join in HNN extension multi-dimensional},
inside this subgroup, for the join $$J=\langle \ldots b_{-2}, b_{-1};\;
b_m, b_{m+1},\ldots\rangle$$ 
of two subgroup $\langle b_m, b_{m+1},\ldots  \rangle$ and  $\langle b_{-1}, b_{{-2}},\ldots  \rangle$,
we have 
$$
\langle b,c\rangle \cap \big\langle 
\langle b,c\rangle^u, \langle b,c\rangle^v
\big\rangle = J.
$$
Hence $J$ is benign, and as the respective finitely generated subgroup $L_J$ of \eqref{EQ benign sample ONE} one may pick the $4$-generator subgroup $L_J = \big\langle 
\langle b,c\rangle^u, \langle b,c\rangle^v
\big\rangle$ of $K_J$.

And to show that $J$ is benign also in $\langle a,b,c\rangle$ just use a slightly larger finitely presented overgroup $\langle a\rangle * K_J$.\end{Example}

Constructions similar to that of
Example~\ref{EX benign sample ONE} are very often used in \cite{A modified proof for Higman, 
On explicit embeddings of Q,
An explicit algorithm} and in related research.

\vskip7mm

\medskip
\noindent 
E-mail:
\href{mailto:v.mikaelian@gmail.com}{v.mikaelian@gmail.com}
$\vphantom{b^{b^{b^{b^b}}}}$

\noindent 
Web: 
\href{https://www.researchgate.net/profile/Vahagn-Mikaelian}{researchgate.net/profile/Vahagn-Mikaelian}

\end{document}